\documentclass[11pt,reqno]{amsart}
\usepackage{amscd,amssymb,amsmath,amsthm}
\usepackage[english]{babel}
\usepackage[arrow,matrix]{xy}
\newtheorem{thm}[subsection]{Theorem}

\newtheorem{pro}[subsection]{Proposition}
\newtheorem{cor}[subsection]{Corollary}

\newtheorem{defn}[subsection]{Definition}

\numberwithin{equation}{section} 

\newcommand{\bea}{\begin{eqnarray}}
\newcommand{\eea}{\end{eqnarray}}

\def \> {\Rightarrow}
\def \0 {\emptyset}

\begin{document}
\sloppy

\title[On Lebesgue Nonlinear Transformations ]{On Lebesgue Nonlinear Transformations}

\author{Nasir Ganikhodjaev}
\address{Nasir Ganikhodjaev \\
 Department of Computational \& Theoretical Sciences\\
Faculty of Science, International Islamic University Malaysia\\
P.O. Box 25200, Kuantan\\
Pahang, Malaysia} \email{{\tt nasirgani@hotmail.com}}
\author{Mansoor Saburov}
\address{Mansoor Saburov \\
Department of Computational \& Theoretical Sciences\\
Faculty of Science, International Islamic University Malaysia\\
P.O. Box 25200, Kuantan\\
Pahang, Malaysia} \email{{\tt msaburovm@gmail.com}}
\author{Ramazon Muhitdinov}
\address{Ramazon Muhitdinov \\
Bukhara State University, Uzbekistan.} \email{{\tt muxitdinov-ramazon@rambler.ru}}

\begin{abstract}
In this paper, we introduce a quadratic stochastic operators on the set of all probability measures of a measurable space.  We study the dynamics of the
Lebesgue quadratic stochastic operator on the set of all Lebesgue measures of the set [0,1]. Namely, we prove the regularity of the Lebesgue quadratic
stochastic operators.

\end{abstract}

\date{\today}
\maketitle

\vskip 0.3 truecm

\section{Introduction}

Quadratic stochastic operator (in short QSO) was first introduced in Bernstein's work \cite{Br}. The QSO was considered an important source of analysis for
 the study of dynamical properties and modeling in various fields such as biology \cite{K,Lyu}, physics \cite{U}, game theory \cite{NGRGUJ}, control system \cite{SMSKh2014a,SMSKh2014b}. Such operator frequently arises in many models of mathematical genetics \cite{NGMSUJ,GNSMNA}. A fixed point set and an omega limiting set of quadratic stochastic operators defined on the finite dimensional simplex were deeply studied in \cite{GR1989,GE} and \cite{MFSM,MFSMQI}. Ergodicity and chaotic dynamics of QSO on the finite dimensional simplex were studied in the papers \cite{NGRGUJ}, \cite{GNSMNA,GZ},
 \cite{MSabu2007,MSabu2013}, and \cite{Z}. In \cite{GRMFRU}, it was given a long self-contained exposition of recent achievements and open problems in
 the theory of quadratic stochastic operators. The analytic theory of stochastic processes generated by quadratic operators was established in \cite{G}.
 In the paper \cite{NGNZH}, the nonlinear Poisson quadratic stochastic operators over the countable state space was studied. In this paper, we shall study
 the  dynamics of Lebesgue quadratic stochastic operators on the set of all Lebesgue measures of the set [0,1]. Let us first recall some notions and
 notations (see \cite{G},\cite{NGNZH}).

Let $(X,\mathbb{F})$ be a measurable space and  $S(X,\mathbb{F})$ be the set of all  probability  measures on $(X,\mathbb{F}),$
where $X$ is a state space and $\mathbb{F}$ is $\sigma$-algebra of subsets of $X.$
It is known that the set  $S(X,\mathbb{F})$ is a compact, convex space and a form of Dirac measure $\delta_x$ which defined by
\begin{displaymath}
\delta_x(A)=\left\{
\begin{array}{rl}
1 & \mbox{ if } x\in A \\
0 & \mbox{ if } x\notin A
\end{array}\right.
\end{displaymath}
for any $A\in\mathbb{B}$  is extremal element of $S(X,\mathbb{B}).$

Let $\{P(x,y,A): x,y\in X, A\in\mathbb{F}\}$  be a family of functions on $X\times X\times \mathbb{F}$
that satisfy the following conditions: \\
i) $P(x,y,\cdot)\in S(X,\mathbb{F}),$  for any  fixed $x,y\in X,$ that is, $P(x,y,\cdot): \mathbb{F}\rightarrow [0,1]$ is the probability measure on
 $\mathbb{F};$  \\
ii) $P(x,y,A)$ regarded as a function of two variables $x$ and $y$ with fixed $A\in\mathbb{F}$ is measurable function on $(X\times X,
\mathbb{F}\otimes\mathbb{F});$ \\
iii) $P(x,y,A)=P(y,x, A)$ for any $x,y\in X, A\in\mathbb{F}.$ \\
We consider a nonlinear transformation (quadratic stochastic operator)
$V: S(X,\mathbb{F})\rightarrow S(X,\mathbb{F})$   defined  by
\begin{equation}\label{qso}
(V\lambda)(A)=\int_X\int_X P(x,y,A)d\lambda(x)d\lambda(x),
\end{equation}
where $\lambda\in S(X,\mathbb{F})$ is an arbitrary
initial probability measure and $A\in \mathbb{F} $ is an arbitrary measurable set.
\begin{defn}
 A probability measure $\mu$ on $(X,\mathbb{F})$ is said to be discrete, if there exists a countable set of elements $\{x_1,x_2,\cdots \}\subset X,$ such that $\mu(\{x_i\})=p_i$ for $i=1,2,\cdots,$ with $\sum_i p_i=1.$ Then $\mu(X\setminus \{x_1,x_2,\cdots\})=0$ and for any $ A\in \mathbb{F}, \mu(A)=\sum_{x_i\in A}\mu(\{x_i\}).$
 \end{defn}
    A family $\{P(x,y,A): x,y\in X, A\in\mathbb{F}\}$ on arbitrary state space $X$, such that for any $x,y\in X$ a measure $P(x,y,\cdot)$ is a discrete measure, is shown in the following example.
   {\bf Example 1 } Let $(X,\mathbb{F})$ be a measurable space. For any $x,y\in X$ and $A\in \mathbb{F},$ assume
   \begin{displaymath}
P(x,y,A)=\left\{
\begin{array}{rl}
0 & \mbox{ if } x\notin A \mbox{ and } y\notin A, \\
\frac{1}{2}  & \mbox{ if } x\in A, y\notin A \mbox{ or } x\notin A, y\in A, \\
1 & \mbox{ if } x\in A \mbox{ and } y\in A.
\end{array}\right.
\end{displaymath}
Assume $\{V^n\lambda: n=0,1,2,\cdots \}$ is the trajectory of the initial point $\lambda\in S(X,\mathbb{F}),$
where $V^{n+1}\lambda=V(V^n\lambda)$ for all $n=0,1,2,\cdots,$ with $ V^0\lambda=\lambda.$
 \begin{defn}
 A quadratic stochastic operator $V$  is called a regular (weak regular), if for any initial measure $\lambda \in S(X,\mathbb{F}),$    the strong limit (respectively weak limit)  $\lim_{n\rightarrow \infty} V^n(\lambda)=\mu$ exists.
 \end{defn}
It is easy to verify that the quadratic stochastic operator $V$ generated by this family is identity transformation, that is for any measure $\lambda\in  S(X,\mathbb{F})$ we have $V\lambda=\lambda.$
In fact, for any $A\in \mathbb{F},$
\begin{eqnarray*}
V\lambda(A)&=&\int_X\int_XP(x,y,A)d\lambda(x)d\lambda(y)= \int_A\int_A 1\cdot d\lambda(x)d\lambda(y) \\
&+&\int_A\int_{A^c}\frac{1}{2}\cdot d\lambda(x)d\lambda(y)+\int_{A^c}\int_A\frac{1}{2}\cdot d\lambda(x)d\lambda(y)+
\int_{A^c}\int_{A^c} 0\cdot d\lambda(x)d\lambda(y) \\
&=& \lambda^2(A)+\frac{1}{2}\lambda(A)(1-\lambda(A))+\frac{1}{2}(1-\lambda(A))\lambda(A)=\lambda(A),
\end{eqnarray*}
where $A^c=X\setminus A.$ \\
If a state space $X=\{1,2,\cdots,m\}$ is a finite set and corresponding $\sigma$-algebra is a power set $\mathcal{P}(X),$ i.e., the set of all subsets of $X,$ then the set of all probability measures on $(X,\mathcal{F})$ has the following form:
 \begin{equation} \label{simplex}
S^{m-1} =\{\textbf{x} = (x_1, x_2,\cdots,x_m)\in R^m :  x_i \geq 0 \mbox{ for any } i, \mbox { and } \sum_{i=1}^m x_i=1\}
\end{equation}
that is called a $(m-1)$-dimensional simplex. \\
In this case for any $i,j\in X$ a probabilistic measure $P(i,j,\cdot)$ is a discrete measure with $\sum_{k=1}^m P(ij,\{k\})=1,$ where $P(ij,\{k\})\equiv P_{ij,k}$ and corresponding qso V has the following form
 \begin{equation} \label{gqso}
(V \textbf{x})_k = \sum_{i,j=1}^m P_{ij,k}x_ix_j
\end{equation}
 for any $\textbf{x} \in S^{m-1}$  and for all $ k = 1,\cdots,m,$  where
 $$a) P_{ij,k}\geq 0, \ \
b) P_{ij,k} = P_{ji,k} \mbox { for all }  i, j, k; \ \   c) \sum_{k=1}^m P_{ij,k}=1.$$
 Such operator can be reinterpreted in terms of evolutionary operator of free population and in this form it has a fair history. Note that  the theory of qso with finite state space is well developed and on the whole all well-known papers devoted to  such qso  \cite{Br}-\cite{NGRGUJ}, \cite{NGMSUJ}-\cite{Z}.
In \cite{NGNZH} the authors studied qso with infinite countable state space. \\
 In this paper we construct the family of  quadratic stochastic operators defined on the continual state space $X=[0,1)$ and investigate their trajectory behavior.
\begin{defn}
 A transformation $V$  is called a Lebesgue qso, if $X=[0,1)$ and $\mathbb{F}$ is a Borel $\sigma$-algebra $\mathbb{B}$ on $[0,1).$
 \end{defn}
In the next section we present a family of Lebesgue qso.
\section{A construction of Lebesgue qso}
Let $X=[0,1)$ and $\mathbb{B}$ is a Borel $\sigma$-algebra on $[0,1).$ For any element $(x,y)\in X\times X,$ we define a discrete probability $P(x,y,\cdot)$ as follows:\\
\begin{eqnarray}
&(i)& \mbox{ for } \ \ x<y \ \ \mbox{ assume } \ \ P(x,y,\{x\})=p \ \ \mbox{ and } \ \ P(x,y,\{y\})=p, \\
&(ii)& \mbox{ for } \ \  x=y \ \ \mbox{ assume } \ \  P(x,x,\{x\})=1, \\
&(iii)& \mbox{ for } \ \ x>y \ \ \mbox{ assume } \ \ P(y,x,\cdot)=P(x,y,\cdot)
\end{eqnarray}
where $p+q=1,$ with $p\geq 0$ and $q\geq 0.$ \\
Let $V$  is a quadratic stochastic operator
\begin{equation}\label{qso}
(V\lambda)(A)=\int_X\int_X P(x,y,A)d\lambda(x)d\lambda(x),
\end{equation}
generated by family of functions (2.1)-(2.3), where $\lambda\in S(X,\mathbb{F})$ is an arbitrary
initial probability measure and $A\in \mathbb{F} $ is an arbitrary measurable set.
 This operator is a natural generalization of Volterra qso. Note that if $p=q=0.5,$ then the corresponding operator is the identity operator. \\
 We show that for any initial measure $\lambda\in S(X,\mathbb{F}),$ there exists strong limit of the sequence  $\{V^n\lambda: n=0,1,2,\cdots \}.$
\section{A Limit behaviour of the trajectories}
 In this section we study the limit behaviour of the trajectory $\{V^n\lambda: n=0,1,2,\cdots \}$ for any initial measure $\lambda\in S(X,\mathbb{F}).$
 \subsection{A discrete initial measure $\lambda$}
 It is easy to verify that for any $a\in [0,1)$ an extremal Dirac measure $\delta_a$ is a fixed point of the operator $V.$ Since $\delta_a(\{a\})=1,$ then from (2.2) we have
 \begin{equation}\label{qso}
(V\delta_a)(\{a\})=\int_X\int_X P(x,y,\{a\})d\delta_a(x)d\delta_a(y)=P(a,a,\{a\})=1=\delta_a(\{a\}),
\end{equation}
that is the Dirac measure $\delta_a$ for any $a\in[0,1)$ is a fixed point. \\
Let a measure $\lambda$  is a convex linear combination of  two Dirac measures $\delta_a$ and $\delta_b$ , i.e., $\lambda=\alpha\delta_a +(1-\alpha)\delta_b,$ where $\alpha\in[0,1]$ and $a,b\in [0,1]$ with  $a<b.$ Simple algebra gives
\begin{equation}\label{qso}
(V\lambda)(\{a\})=\int_X\int_X P(x,y,\{a\})d\lambda(x)d\lambda(y)=\lambda(a)[\lambda(a)+2p\lambda(b)],
\end{equation}
and
\begin{equation}\label{qso}
(V\lambda)(\{b\})=\int_X\int_X P(x,y,\{b\})d\lambda(x)d\lambda(y)=\lambda(b)[\lambda(b)+2q\lambda(a)],
\end{equation}
i.e., $V\lambda$ is the convex linear of the same two Dirac measures $\delta_a$ and $\delta_b$ with  $$V\lambda=\alpha_1\delta_a +(1-\alpha_1)\delta_b,$$
where $\alpha_1=\alpha[\alpha+2p(1-\alpha)]$ and $1-\alpha_1=(1-\alpha)[1-\alpha+2q\alpha].$ Then it is evident that $V^2\lambda$ is the convex linear of the same two Dirac measures $\delta_a$ and $\delta_b$ with  $$V^2\lambda=\alpha_2\delta_a +(1-\alpha_2)\delta_b,$$
where $\alpha_2=\alpha_1[\alpha_1+2p(1-\alpha_1)]$ and $1-\alpha_2=(1-\alpha_1)[1-\alpha_1+2q\alpha_1].$ Thus one can show that $V^n\lambda$ is the convex linear of the same two Dirac measures $\delta_a$ and $\delta_b$ with  $$V^n\lambda=\alpha_n\delta_a +(1-\alpha_n)\delta_b,$$
where $\alpha_n=\alpha_{n-1}[\alpha_{n-1}+2p(1-\alpha_{n-1})]$ and $1-\alpha_n=(1-\alpha_{n-1})[1-\alpha_{n-1}+2q\alpha_{n-1}].$
Simple calculus gives
\begin{displaymath}
\lim_{n\rightarrow \infty}\alpha_n=\left\{
\begin{array}{rl}
1 & \mbox{ if } p>\frac{1}{2} \\
0 & \mbox{ if } p<\frac{1}{2}
\end{array}\right.
\end{displaymath}
that is
\begin{displaymath}
\lim_{n\rightarrow \infty}V^n\lambda=\left\{
\begin{array}{rl}
\delta_a & \mbox{ if } p>\frac{1}{2} \\
\delta_b & \mbox{ if } p<\frac{1}{2}
\end{array}\right.
\end{displaymath}
Let now a measure $\lambda$  is a convex linear combination of  three Dirac measures $\delta_a,\delta_b$ and $\delta_c,$  i.e., $\lambda=\alpha\delta_a +\beta\delta_b+(1-\alpha-\beta)\delta_c,$ where $\alpha,\beta\in[0,1]$ and $a,b,c\in [0,1]$ with $\alpha+\beta\leq1$ and $a<b<c.$ Simple algebra gives
\begin{equation}\label{qso}
(V\lambda)(\{a\})=\int_X\int_X P(x,y,\{a\})d\lambda(x)d\lambda(y)=\lambda(a)[\lambda(a)+2p(\lambda(b)+\lambda(c)],
\end{equation}
\begin{equation}\label{qso}
(V\lambda)(\{b\})=\int_X\int_X P(x,y,\{a\})d\lambda(x)d\lambda(y)=\lambda(b)[\lambda(b)+2q\lambda(a)+2p\lambda(c)],
\end{equation}
and
\begin{equation}\label{qso}
(V\lambda)(\{c\})=\int_X\int_X P(x,y,\{a\})d\lambda(x)d\lambda(y)=\lambda(c)[\lambda(c)+2q(\lambda(a)+\lambda(b)],
\end{equation}
i.e., $V\lambda$ is the convex linear of the same three Dirac measures $\delta_a,\delta_b$ and $\delta_c$ with  $$V\lambda=\alpha_1\delta_a +\beta_1\delta_b+(1-\alpha_1-\beta_1)\delta_c,$$
where $\alpha_1=\alpha[\alpha+2p(1-\alpha)],\beta_1=\beta[\beta+2q\alpha+2p(1-\alpha-\beta)]$ and $1-\alpha_1-\beta_1=(1-\alpha-\beta)[1-\alpha-\beta+2q(1-\alpha-\beta].$ Then it is evident that $V^2\lambda$ is the convex linear of the same three Dirac measures $\delta_a,\delta_b$ and $\delta_c,$  with $V^2\lambda=\alpha_2\delta_a +\beta_2\delta_b+(1-\alpha_2-\beta_2)\delta_c,$
where $\alpha_2=\alpha_1[\alpha_1+2p(1-\alpha_1)],\beta_2=\beta_1[\beta_1+2q\alpha_1+2p(1-\alpha_1-\beta_1)]$ and $1-\alpha_2-\beta_2=(1-\alpha_1-\beta_1)[1-\alpha_1-\beta_1+2q(1-\alpha_1-\beta_1].$
Thus one can show that $V^n\lambda$ is the convex linear of the same three Dirac measures $\delta_a,\delta_b$ and $\delta_c$ with  $$V^n\lambda=\alpha_n\delta_a +\beta_n\delta_b+(1-\alpha_n-\beta_n)\delta_c,$$
where $\alpha_n=\alpha_{n-1}[\alpha_{n-1}+2p(1-\alpha_{n-1})],\beta_n=\beta_{n-1}[\beta_{n-1}+2q\alpha_{n-1}+2p(1-\alpha_{n-1}-\beta_{n-1})]$ and $1-\alpha_n-\beta_n=(1-\alpha_{n-1}-\beta_{n-1})[1-\alpha_{n-1}-\beta_{n-1}+2q(1-\alpha_{n-1}-\beta_{n-1}].$
As noted above if $p>\frac{1}{2}$ then $\lim_{n\rightarrow \infty}\alpha_n=1$ and if $p<\frac{1}{2}$ then $\lim_{n\rightarrow \infty}(1-\alpha_n-\beta_n)=1$
that is
\begin{displaymath}
\lim_{n\rightarrow \infty}V^n\lambda=\left\{
\begin{array}{rl}
\delta_a & \mbox{ if } p>\frac{1}{2} \\
\delta_c & \mbox{ if } p<\frac{1}{2}
\end{array}\right.
\end{displaymath}

Let a measure $\lambda$  is a convex linear combination of  $n$  Dirac measures $\{\delta_{a_i}, i=1,\cdots,n\}$  i.e., $\lambda=\sum_{i=1}^n\alpha_i\delta_{a_i},$ where $\alpha_i\in[0,1], i=1,\cdots,n$ and $a_i\in [0,1], i=1,\cdots,n$ with $\sum_{i=1}^n\alpha=1$ and $a_1<a_2<\cdots<a_n.$ Simple algebra gives
\begin{equation}\label{qso}
(V\lambda)(\{a_1\})=\int_X\int_X P(x,y,\{a_1\})d\lambda(x)d\lambda(y)=\lambda(a_1)[\lambda(a_1)+2p(1-\lambda(a_1))],
\end{equation}
and
\begin{equation}\label{qso}
(V\lambda)(\{a_n\})=\int_X\int_X P(x,y,\{a_n\})d\lambda(x)d\lambda(y)=\lambda(a_n)[\lambda(a_n)+2q(1-\lambda(a_n))],
\end{equation}
As shown above we have
\begin{displaymath}
\lim_{n\rightarrow \infty}V^n\lambda=\left\{
\begin{array}{rl}
\delta_{a_1} & \mbox{ if } p>\frac{1}{2} \\
\delta_{a_n} & \mbox{ if } p<\frac{1}{2}
\end{array}\right.
\end{displaymath}
\subsection{A continuous initial measure $\lambda$} Let $\lambda\in S(X,\mathbb{B})$ be a continuous probability measure and $A=[a,b]\in \mathbb{B}$ be a segment
in $X$ with $A^c=[0,a)\cup (b,1).$ Then,
\begin{eqnarray*}
V\lambda(A)&=&\int_a^b\int_a^b 1\cdot d\lambda(x)d\lambda(y)+ \int_a^b\int_0^a p\cdot d\lambda(x)d\lambda(y)
+\int_a^b\int_b^1q\cdot d\lambda(x)d\lambda(y)\\
&+&\int_0^a\int_a^b p\cdot d\lambda(x)d\lambda(y)+\int_b^1\int_a^b q\cdot d\lambda(x)d\lambda(y)+ \int_0^a\int_0^a 0\cdot d\lambda(x)d\lambda(y)\\
&+& \int_0^a\int_b^10\cdot d\lambda(x)d\lambda(y)+\int_b^1\int_0^a 0\cdot d\lambda(x)d\lambda(y)+\int_b^1\int_b^1 0\cdot d\lambda(x)d\lambda(y)\\
&=&\lambda([a,b])\left[\lambda([a,b])+2p\lambda([0,a))+2q\lambda((b,1))\right].
\end{eqnarray*}
It is evident that the measure $V\lambda$ is absolutely continuous with respect to measure $\lambda.$ Then according Radon-Nikodym Theorem, there exists non-negative measurable function $f_\lambda^{(1)}:X\rightarrow R$ called a density, such that
\begin{equation}
V\lambda(A)=\int_Af_\lambda^{(1)}(x)d\lambda(x).
\end{equation}
The derivations of the density functions are presented as follows. For rather small segment $[x,x+\Delta x]$ we have
\begin{equation}
V\lambda([x,x+\Delta x])=\lambda([x,x+\Delta x])[\lambda([x,x+\Delta x])+2p\lambda([0,x))+2q\lambda((x+\Delta x,1))]
\end{equation}
and
\begin{eqnarray*}
f_\lambda^{(1)}(x)&=& \lim_{\Delta x\rightarrow 0} \frac{V\lambda([x,x+\Delta x])}{\lambda([x,x+\Delta x])}\\
&=&\lim_{\Delta x\rightarrow 0}[\lambda([x,x+\Delta x])+2p\lambda([0,x))+2q\lambda((x+\Delta x,1))]\\
&=&2p\lambda([0,x))+2q\lambda((x,1)).
\end{eqnarray*}
Now consider a measure $V^2\lambda=V(V\lambda).$ It is evident that
\begin{equation}
V^2\lambda(A)=\int_Af_{V\lambda}^{(1)}(x)dV\lambda(x).
\end{equation}
and since $V^2\lambda$ is absolutely continuous with respect to measure $\lambda,$ we have
\begin{equation}
V^2\lambda(A)=\int_Af_\lambda^{(2)}(x)d\lambda(x).
\end{equation}
According (2.6) we have
\begin{eqnarray*}
\begin{split}
&V^2\lambda([x,x+\Delta x]) \\
&=V\lambda([x,x+\Delta x])[V\lambda([x,x+\Delta x])+2pV\lambda([0,x))+2qV\lambda((x+\Delta x,1))]\\
&=\lambda([x,x+\Delta x])[\lambda([x,x+\Delta x])+2p\lambda([0,x))+2q\lambda((x+\Delta x,1))]\\
&\quad \quad \quad \cdot\{\lambda([x,x+\Delta x])[\lambda([x,x+\Delta x])+2p\lambda([0,x))+2q\lambda((x+\Delta x,1))]\\
&\quad \quad \quad \quad +2p\lambda([0,x])[\lambda([0,x])+2q\lambda((x,1))]+\\
& \quad \quad \quad \quad +2q\lambda([x+\Delta x,1))[\lambda([x+\Delta x,1))+2p\lambda([0,x+\Delta x))\}
\end{split}
\end{eqnarray*}
Then
\begin{eqnarray*}
f_\lambda^{(2)}(x)&=& [2p\lambda([0,x))+2q\lambda([x,1))]\\
&&\cdot\{2p\lambda([0,x])[\lambda([0,x))+2q\lambda([x,1))]+2q\lambda([x,1))[\lambda([x,1))+2p\lambda([0,x])\}.
\end{eqnarray*}
Similarly, one can show that a measure $V^n\lambda$ is absolutely continuous with respect to $\lambda$ for any $n$ and
\begin{equation}\label{Vn}
V^n\lambda(A)=\int_Af_\lambda^{(n)}(x)d\lambda(x).
\end{equation}
Let $g_\lambda(x)=\lambda([0,x))$ and $g_\lambda^{(n)}(x)=g_\lambda^{(n-1)}(x)(g_\lambda^{(n-1)}(x)+2q(1-g_\lambda^{(n-1)}(x)),$ for $n=1,2,3,\cdots $, where
$g_\lambda^{(0)}(x)=x$ and $g_\lambda^{(1)}(x)=g_\lambda(x).$ It is evident that
 $$1-g_\lambda^{(n)}(x)=(1-g_\lambda^{(n-1)}(x))(1-g_\lambda^{(n-1)}(x)+2pg_\lambda^{(n-1)}(x)).$$
Then, since $\lambda([x,1))=1-\lambda([0,x)),$ we have
\begin{equation}
f(x)=2px+2q(1-x),
\end{equation}
\begin{equation}
f_\lambda^{(1)}(x)=f(g_\lambda^{(1)}(x)),
\end{equation}
and
\begin{equation}
f_\lambda^{(2)}(x)=f(g_\lambda^{(1)}(x))\cdot f(g_\lambda^{(2)}(x)).
\end{equation}
Using induction, one can prove that for any $n$ we have
\begin{equation}\label{fn}
f_\lambda^{(n)}(x)=\prod_{i=1}^nf(g_\lambda^{(i)}(x)).
\end{equation}
It is easy to see that $f_\lambda^{(n)}(0)=(2q)^n$ and $f_\lambda^{(n)}(1)=(2p)^n.$
Since
$$\int_0^1f_\lambda^{(n)}(x)d\lambda(x)=1,$$
one has that
\begin{equation}
f_\lambda^{(n)}(0)\rightarrow 0 \ \ \mbox{and } \ \ f_\lambda^{(n)}(1)\rightarrow \infty \ \ \mbox{ if } p>1/2,
\end{equation}
and
\begin{equation}
f_\lambda^{(n)}(0)\rightarrow \infty \ \ \mbox{ and } \ \  f_\lambda^{(n)}(1)\rightarrow 0 \ \ \mbox{ if }
 p<1/2.
\end{equation}
If $\lambda=m$ is a usual Lebesgue measure on $[0,1],$ then $m([0,x))=x$ and $g_m^{(1)}(x)=x.$ In this case one can explicitly find the functions $f_m^{(n)}(x)$ for any $n.$ \\

\section{Regularity of Lebesgue qso}

Now, we are aiming to study the limit behavior of the Radon-Nikodym derivatives $f_\lambda^{(n)}(\cdot)$ for $n\rightarrow \infty.$ Let $G(x)=x(x+2q(1-x))$ for $p,q\geq 0$ and $p+q=1.$ We always assume that $p,q\neq \frac{1}{2}$. One can easily check that $G'(x)=f(x)$ and $G''(x)=f'(x)=2(p-q).$ Since $f(x)\geq 0$ for any $x\in[0,1],$ the function $G:[0,1]\to[0,1]$ is increasing. Moreover, the function $f:[0,1]\to\mathbb{R}_{+}$ is increasing whenever $p>q$ (or equivalently $p>\frac{1}{2}$) and decreasing whenever $p<q$ (or equivalently $p<\frac{1}{2}$).

It is easy to check that $g^{(n)}_\lambda(x)=G\left(g^{(n-1)}_\lambda(x)\right)$. We know that $g_\lambda:[0,1]\to[0,1],$ $g_\lambda(x)=\lambda([0,x))$ is the increasing function. We suppose that the function $g_\lambda:[0,1]\to[0,1]$ is also differentiable. It is clear that
$$
(g_\lambda^{(n)}(x))'=G'\left(g^{(n-1)}_\lambda(x)\right)\cdot (g^{(n-1)}_\lambda(x))'
$$
Consequently, $g^{(n)}_\lambda:[0,1]\to[0,1],$ for all $n\in\mathbb{N},$ are differentiable and increasing functions.

\begin{pro}
Let $f_\lambda^{(n)}:[0,1]\to\mathbb{R}_{+},$ $n\in\mathbb{N}$ be functions given by \eqref{fn}. Then, the function $f_\lambda^{(n)}$ is increasing whenever $p>\frac{1}{2}$ and decreasing whenever $p<\frac{1}{2}$.
\end{pro}
\begin{proof}
We know that
$$
(f_\lambda^{(n)}(x))'=\sum\limits_{k=1}^{n}\left(\prod\limits_{i=1,\ i\neq k}^{n}f(g^{i})_\lambda(x)\right)\cdot f'(g^{(k)}_\lambda(x))\cdot \left(g^{(k)}_\lambda(x)\right)'.
$$
Since $\prod\limits_{i=1,\ i\neq k}^{n}f(g^{i})_\lambda(x)\geq 0$ and  $\left(g^{(k)}_\lambda(x)\right)'\geq 0,$ the function  $f_\lambda^{(n)}$ is increasing whenever $p>\frac{1}{2}$ and decreasing whenever $p<\frac{1}{2}.$ This completes the proof.
\end{proof}

 Let $p>\frac{1}{2}$ and $\beta_n=\frac{1-\frac{1}{n}}{\sqrt[n-1]{16p^4}}$. It is clear that $0<\beta_n<1$ and $\lim\limits_{n\to\infty}\beta_n=1$. So, since $2q<1,$ one has that $\beta_n>2q$ for large $n$. Moreover,  we have that $2p\sqrt{\beta_n^{n-1}}<\frac{1}{2p}<1$.

 Let $B_n=\frac{\beta_n-2q}{1-2q}.$ One has that $0<B_n<1$ and $\lim\limits_{n\to\infty}B_n=1.$

 \begin{pro}
Let $p>\frac{1}{2}$ and $\beta_n, B_n$ be given as above. The following statements hold true:
 \begin{itemize}
 \item[$(i)$] One has that $G(x)\leq \beta_nx$ for any $x\in[0,B_n];$
 \item[$(ii)$] One has that $G[0,B_n]=[0,\beta_nB_n]\subset[0,B_n].$
 \end{itemize}
 \end{pro}
 \begin{proof}
 Since $0\leq x\leq B_n=\frac{\beta_n-2q}{1-2q},$ we have that $x+2q(1-x)\leq \beta_n$ or equivalently $G(x)\leq \beta_nx$. On the other hand, since $G(x)$ is the increasing function, we get that $0=G(0)\leq G(x)\leq G(B_n)=\beta_nB_n<B_n.$
 \end{proof}
\begin{cor}\label{gi}
Let $p>\frac{1}{2}$. One has that $0\leq g^{(i)}_\lambda(x)\leq \beta_n^{i-1}g_\lambda(x), \ \forall \ x\in[0,B_n], i=\overline{2,n}$.
\end{cor}
\begin{thm}\label{p>0.5}
Let $p>\frac{1}{2}.$ One has that $0\leq f_\lambda^{(n)}(x)\leq (\frac{1}{2p})^n$ for any $x\in[0,B_n].$
\end{thm}
\begin{proof}
We know that $f_\lambda^{(n)}(x)=\prod\limits_{i=1}^nf(g_\lambda^{(i)}(x)).$ Due to Corollary \ref{gi}, since $f$ is the increasing linear function, we then obtain for any $x\in[0,B_n]$ that
$$
f_\lambda^{(n)}(x)\leq \prod\limits_{i=1}^n\beta_n^{i-1}f(g_\lambda(x))=\left(\sqrt{\beta_n^{n-1}}f(g_\lambda(x))\right)^n.
$$
Since $2q\leq f(g_\lambda(x))\leq 2p,$ we have that $0\leq f_\lambda^{(n)}(x)\leq \left(2p\sqrt{\beta_n^{n-1}}\right)^n\leq \left(\frac{1}{2p}\right)^n$ for any $x\in[0,B_n]$. This completes the proof.
\end{proof}

Similarly, one can prove the following result.

\begin{thm}\label{p<0.5}
Let $p<\frac{1}{2}.$ There exists $0<A_n<1$ such that $\lim\limits_{n\to\infty}A_n=0$ and $0\leq f_\lambda^{(n)}(x)\leq (\frac{1}{2q})^n$ for any $x\in[A_n,1].$
\end{thm}

Hence, the sequence of functions $f_\lambda^{(n)}(x)$ has a tall spike at the end points of the segment $[0,1]$ whenever $p,q\neq\frac{1}{2}.$ Consequently, the (weak) limit of this sequence of the functions $f_\lambda^{(n)}(x)$ is the Dirac delta function concentrated at the end points of the segment $[0,1]$.

\begin{thm}
Let $V$ be a qso generated by family of functions (2.1)-(2.3). Then for any initial continuous measure $\lambda\in S(X, \mathbb{B})$ there exists a strong limit of the sequence of measures $\{V^k\lambda\}$ where
\begin{displaymath}
\lim_{n\rightarrow \infty}V^n\lambda=\left\{
\begin{array}{rl}
\delta_0 & \mbox{ if } p<\frac{1}{2} \\
\delta_1 & \mbox{ if } p>\frac{1}{2}.
\end{array}\right.
\end{displaymath}
\end{thm}

Recall that if $p=q=\frac{1}{2},$ then the corresponding qso is the identity transformation.

\begin{cor}
The Lebesgue quadratic stochastic operator $V$ generated by family of functions (2.1)-(2.3) is a regular transformation.
\end{cor}

\section*{Acknowledgement}
The work has been supported by the MOHE grant FRGS14-116-0357. The third author (R.M.) wishes to thank International Islamic University Malaysia (IIUM), where this paper was written, for the invitation and hospitality.

\end{document}